\documentclass[11pt,a4paper]{article}
\usepackage[top=2.5cm,bottom=2.5cm,left=2.5cm,right=2.5cm]{geometry}

\usepackage{amssymb}
\usepackage{amssymb,amsmath,graphicx,color,amsthm}

\newtheorem{theorem}{Theorem}

\newtheorem{corollary}[theorem]{Corollary}
\newtheorem{proposition}[theorem]{Proposition}

\newtheorem{observation}{Observation}
\newtheorem{conjecture}{Conjecture}
\newtheorem{question}{Question}

\newcommand{\set}[1]{\ensuremath{\left\{#1 \right\}}}

\newcommand{\cartp}{\mathbin{\square}}

\usepackage{url}
\makeatletter
\g@addto@macro{\UrlBreaks}{\UrlOrds}
\makeatother

\begin{document}

\title{On incidence coloring conjecture in \\Cartesian products of graphs}

\author
{
	Petr Gregor\thanks{Department of Theoretical Computer Science and Mathematical Logic, Charles University, Prague, Czech Republic. 
		E-Mail: \texttt{gregor@ktiml.mff.cuni.cz}}, \
	Borut Lu\v{z}ar\thanks{Faculty of Information Studies, Novo mesto, Slovenia. E-Mail: \texttt{borut.luzar@gmail.com}}, \
	Roman Sot\'{a}k\thanks{Faculty of Science, Pavol Jozef \v Saf\'arik University, Ko\v sice, Slovakia.
		E-Mail: \texttt{roman.sotak@upjs.sk}}
}

\maketitle

{
\begin{abstract}
	An \textit{incidence} in a graph $G$ is a pair $(v,e)$ where $v$ is a vertex of $G$ 
	and $e$ is an edge of $G$ incident to $v$. 
	Two incidences $(v,e)$ and $(u,f)$ are \textit{adjacent} if at least one of the following holds: 
	$(a)$ $v = u$, $(b)$ $e = f$, or $(c)$ $vu \in \{e,f\}$. An \textit{incidence coloring}
	of $G$ is a coloring of its incidences assigning distinct colors to adjacent incidences. 
	It was conjectured that at most $\Delta(G) + 2$ colors are needed for an incidence coloring
	of any graph $G$. The conjecture is false in general, but the bound holds for many classes 
	of graphs. We introduce some sufficient properties of the two factor graphs of a Cartesian
	product graph $G$ for which $G$ admits an incidence coloring with at most $\Delta(G) + 2$ colors.	
\end{abstract}
}


\section{Introduction}

An \textit{incidence} in a graph $G$ is a pair $(v,e)$ where $v$ is a vertex of $G$ 
and $e$ is an edge of $G$ incident to $v$. The set of all incidences of $G$ is denoted $I(G)$.
Two incidences $(v,e)$ and $(u,f)$ are \textit{adjacent} if 
at least one of the following holds: $(a)$ $v = u$, $(b)$ $e = f$, or $(c)$ $vu \in \{e,f\}$ (see Fig.~\ref{fig:def}). 
An \textit{incidence coloring} of $G$ is a coloring of its incidences such that adjacent incidences are assigned 
distinct colors. The least $k$ such that $G$ admits an incidence coloring with $k$ colors is called the 
\textit{incidence chromatic number of $G$}, denoted by $\chi_i(G)$. An incidence coloring of $G$ is called \textit{optimal} 
if it uses precisely $\chi_i(G)$ colors.
\begin{figure}[ht]
	$$
		\includegraphics[scale=1.0]{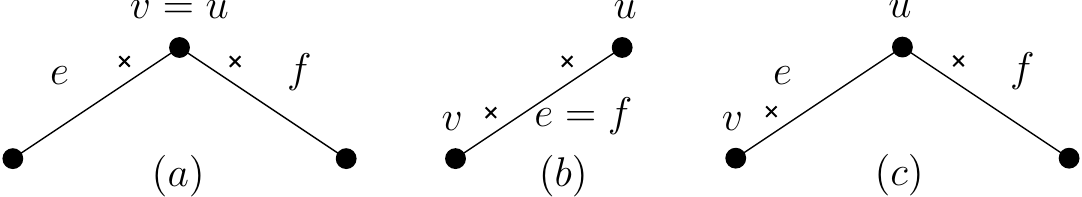}
	$$
	\caption{Three types of adjacent incidences.}
	\label{fig:def}
\end{figure}

The incidence coloring of graphs was defined in 1993 by Brualdi and Massey~\cite{BruMas93} and attracted considerable attention as 
it is related to several other types of colorings. As already observed by the originators, it is directly connected to \textit{strong edge-coloring}, 
i.e. a proper edge-coloring such that the edges at distance at most two receive distinct colors. Indeed, consider a bipartite graph $H$ with the vertex set 
$V(H)= V(G)\cup E(G)$ and two vertices $v\in V(G)$ and $e \in E(G)$ adjacent in $H$ if and only if $v$ is incident to $e$ in $G$; that is, $H$ is the 
graph $G$ with every edge subdivided. Then, a strong edge-coloring of $H$ corresponds to an incidence coloring of $G$.
This in turn means that the incidence chromatic number of a graph $G$ is equal to the strong chromatic index of $H$.

A graph $G$ is called a \textit{($\Delta+k$)-graph} if it admits an incidence coloring with $\Delta(G)+k$ colors for some positive integer $k$.
A complete characterization of $(\Delta+1)$-graphs is still an open problem.
While it is a trivial observation that complete graphs and trees are such, it is harder to determine additional classes.
This problem has already been addressed in several papers; it was shown that
Halin graphs with maximum degree at least $5$~\cite{WanChePan02}, outerplanar graphs with maximum degree at least $7$~\cite{ShiSun08}, 
planar graphs with girth at least $14$~\cite{BonLevPin11}, and square, honeycomb and hexagonal meshes~\cite{HuaWanChu04} are $(\Delta+1)$-graphs. 
In fact, every $n$-regular graph with a partition into $n+1$ (perfect) dominating sets is such, as observed by Sun~\cite{Sun12}.
\begin{theorem}[Sun, 2012]\label{thm:partition}
	If $G$ is an $n$-regular graph, then $\chi_i(G) = n + 1$ if and only if $V(G)$
    is a disjoint union of $n+1$ dominating sets.
\end{theorem}

\begin{observation}
	\label{obs:reg} 
  In any optimal incidence coloring $c$ of a regular ($\Delta+1$)-graph $G$, 
  for every vertex $v$ there is a color $c_v$ such that for every edge $uv$,  it holds $c(u,uv)=c_v$. 
\end{observation}

Similarly intriguing as the lower bound is the upper bound. Brualdi and Massey~\cite{BruMas93} proved that 
$\chi_i(G) \le 2\Delta(G)$ for every graph $G$. Aside to that, they proposed the following.
\begin{conjecture}[Brualdi and Massey, 1993]
	\label{conj:main}
    For every graph $G$, $$\chi_i(G) \le \Delta(G) + 2\,.$$
\end{conjecture}
Conjecture~\ref{conj:main} has been disproved by Guiduli~\cite{Gui97} who observed that incidence coloring is a special case of directed star arboricity, 
introduced by Algor and Alon~\cite{AlgAlo89}. Based on this observation, it was clear that Paley graphs are counterexamples to the conjecture. 
Nevertheless, Guiduli~\cite{Gui97} decreased the upper bound for simple graphs.
\begin{theorem}[Guiduli, 1997]
	For every simple graph $G$, 
    $$
    	\chi_i(G) \le \Delta(G) + 20 \log \Delta(G) + 84\,.
    $$
\end{theorem}

Although Conjecture~\ref{conj:main} has been disproved in general, it has been confirmed for many graph classes, 
e.g. cubic graphs~\cite{May05}, partial $2$-trees (and thus also outerplanar graphs)~\cite{DolSopZhu04}, and powers of cycles 
(with a finite number of exceptions)~\cite{NakNak12}, to list just a few.
We refer an interested reader to~\cite{Sop15} for a thorough online survey on incidence coloring results.

Recently, Pai et al.~\cite{PaiChaYanWu14} considered incidence coloring of hypercubes. Recall that the $n$-dimensional 
hypercube $Q_n$ for an integer $n\ge 1$ is the graph with the vertex set $V(Q_n)=\{0,1\}^n$ and edges between vertices that differ in exactly one coordinate. They proved
that the conjecture holds for hypercubes of special dimensions in the following form.
\begin{theorem}[Pai et al., 2014]
	For every integers $p,q\ge 1$,
    \begin{itemize}
    	 \item[$(i)$] $\chi_i(Q_n) = n + 1$, if $n = 2^p - 1$;
         \item[$(ii)$] $\chi_i(Q_n) = n + 2$, 
         	if $n = 2^p - 2$ and $p\ge 2$, or $n = 2^p + 2^q - 1$, or $n = 2^p + 2^q - 3$ and $p,q \ge 2$.
    \end{itemize}
\end{theorem}
They also obtained some additional upper bounds for the dimensions of other forms and proposed the following conjecture. 
\begin{conjecture}[Pai et al., 2014]
	\label{conj:pai}
	For every $n\ge 1$ such that $n \neq 2^p - 1$ for every integer $p \ge 1$, 
    $$
    	\chi_i(Q_n) = n + 2\,.
    $$
\end{conjecture}

Motivated by their research, we consider incidence coloring of Cartesian products of graphs; in particular we study
sufficient conditions for the factors such that their Cartesian product is a $(\Delta+2)$-graph.
Conjecture~\ref{conj:pai} has recently been confirmed also by Shiau, Shiau and Wang~\cite{ShiShiWan15}, 
who also considered Cartesian products of graphs, but in a different setting, not limiting to $(\Delta+2)$-graphs.

In Section~\ref{sec:1}, we show that if one of the factors is a $(\Delta+1)$-graph and the maximum degree of the second is not too small, 
Conjecture~\ref{conj:main} holds. As a corollary, Conjecture~\ref{conj:pai} is also answered in affirmative. In Section~\ref{sec:2}, we introduce
two classes of graphs, $2$-permutable and $2$-adjustable, such that the Cartesian product of factors from each of them is a $(\Delta+2)$-graph. In Section~\ref{sec:con}, 
we discuss further work and propose several problems and conjectures.

\section{Preliminaries}
\label{sec:prel}

In this section, we present additional terminology used in the paper.
The \textit{Cartesian product of graphs $G$ and $H$}, denoted by $G \cartp H$, is the graph with the vertex set $V(G) \times V(H)$ and edges between vertices $(u,v)$ and $(u',v')$ whenever
\begin{itemize}
\item $uu'\in E(G)$ and $v=v'$ (a \emph{$G$-edge}), or
\item $u=u'$ and $vv'\in E(H)$ (an \emph{$H$-edge}).
\end{itemize}
We call the graphs $G$ and $H$ the \textit{factor graphs}.
The \emph{$G$-f\mbox{}iber} with respect to $v\in V(H)$, denoted by $G_v$, is the copy of $G$ in $G \cartp H$ induced by the 
vertices with the second component $v$. Analogously, the \emph{$H$-f\mbox{}iber} with respect to $u \in V(G)$, 
denoted by $H_u$, is the copy of $H$ in $G \cartp H$ induced by the vertices with the first component $u$. 
An incidence of a vertex in $G \cartp H$ and a $G$-edge (resp. an $H$-edge) is called a \emph{$G$-incidence} 
(resp. an $H$-incidence). Clearly, $G$-incidences are in $G$-f\mbox{}ibers and $H$-incidences are in $H$-f\mbox{}ibers.

Let $c$ be an incidence coloring of a graph $G$. 
The \emph{$c$-spectrum} (or simply the \textit{spectrum}) of a vertex $v\in V(G)$ is the set $S_c(v)$ of all colors assigned 
to the incidences of the edges containing $v$; that is,
$$
	S_c(v)=\{c(v,uv), c(u,uv) \mid uv \in E(G)\}.
$$
Using the size of vertex spectra, a trivial lower bound for the incidence chromatic number of a graph $G$ is obtained:
$$
	\chi_i(G) \ge \min_c \max_{v \in V(G)} |S_c(v)|\,.
$$
Clearly, for every vertex $v$ of $G$, every incidence in $I(v) = \set{(v,uv) \ | \ uv \in E(G)}$ 
must have distinct colors assigned. Moreover, for every edge $vu$ the color of the incidence $(u, vu)$ must be different from all colors of $I(v)$. 
Thus, every spectrum $S_c(v)$ of a non-isolated vertex $v$ is of size at least $d(v) + 1$, implying that $\chi_i(G) \ge \Delta(G) + 1$, 
where $d(v)$ denotes the degree of $v$ and $\Delta(G)$ denotes the maximal degree in $G$. 
We distinguish two types of colors from a vertex's $v$ perspective: by $S_c^0(v)$ and $S_c^1(v)$ we denote the 
sets $\{c(v,uv) \mid uv \in E(G)\}$ and $\{c(u,uv) \mid uv \in E(G)\}$, respectively. 
Clearly, $S_c(v) = S_c^0(v) \mathop{\dot{\cup}} S_c^1(v)$ (disjoint union), $|S_c^0(v)| = d(v)$ and $|S_c^1(v)| \ge 1$ if $v$ is a non-isolated vertex.

In the sequel we will also use the following version of incidence colorings. 
A \textit{$(k,p)$-incidence coloring} of a graph $G$ is a $k$-incidence coloring $c$ of $G$ 
with $|S_c^1(v)| \le p$ for every $v \in V(G)$. The smallest $k$ for which a $(k,p)$-incidence 
coloring of a graph $G$ exists is denoted $\chi_{i,p}(G)$, and clearly, 
$\chi_i(G) \le \chi_{i,p}(G)$ for every positive integer $p$. 
We say that a graph is a \textit{$(k,p)$-graph} if it admits a $(k,p)$-incidence coloring. 

Observe that every regular $(\Delta+1)$-graph $G$ admits also a $(\Delta(G)+1,1)$-incidence coloring (see Observation~\ref{obs:reg}).
As noted by Sopena~\cite{Sop15}, a $(k,1)$-incidence coloring of a graph $G$ is equivalent to a 
vertex coloring of the square $G^2$, and so
\begin{equation}
	\label{eq:square}
	\chi_i(G) \le \chi_{i,1}(G) = \chi(G^2)\,.
\end{equation}
On the other hand, an irregular $(\Delta+1)$-graph $G$ does not necessarily admit a $(\Delta(G)+1,1)$-incidence coloring. For example, 
let $H$ be a cycle $C_5$ with a pendant edge, which is a $(\Delta+1)$-graph. 
It has $\chi_{i}(H) = \chi_{i,2}(H) = 4$, but $\chi_{i,1}(H) = 5$ (see Fig.~\ref{fig:C5e} for example). In fact, the difference $\chi_{i,1}(G) - \chi_{i}(G)$
can be arbitrarily large. Consider simply balanced complete bipartite graphs $K_{n,n}$, which are $(\Delta+2)$-graphs~\cite{ShiSun08},
whereas $\chi_{i,1}(K_{n,n}) = \chi(K_{n,n}^2) = 2n$. 
\begin{figure}[ht]
	$$
		\includegraphics[scale=0.8]{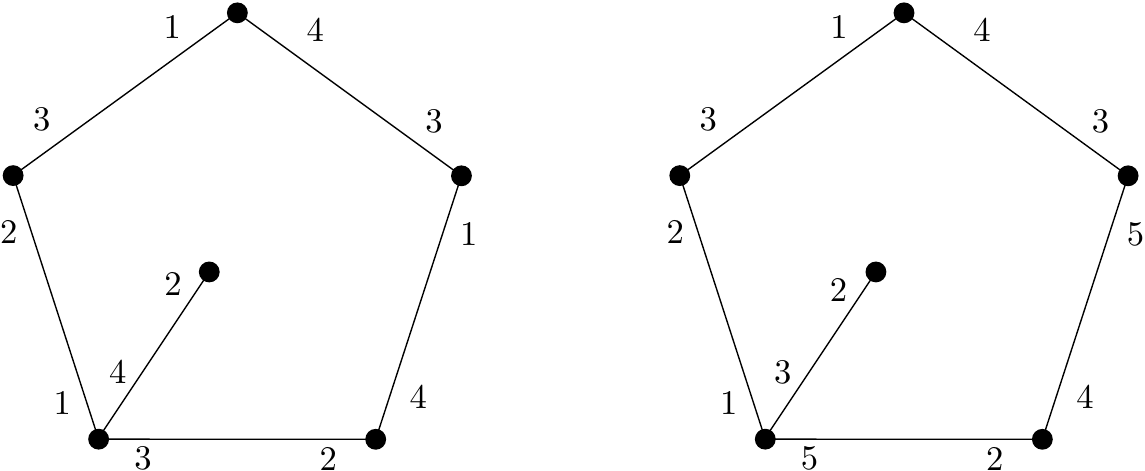}
	$$
	\caption{A $(4,2)$-incidence coloring (left) and a $(5,1)$-incidence coloring (right) of a cycle on five vertices with a pendant edge.}
	\label{fig:C5e}
\end{figure}

\section{Cartesian products with $(\Delta+1)$-graphs}
\label{sec:1}

In this section, we present some sufficient conditions for a Cartesian product of graphs to be a $(\Delta+2)$-graph 
when one of the factors is a $(\Delta+1)$-graph. 

\begin{theorem}
  \label{thm:product}
  Let $G$ be a ($\Delta+1$)-graph and let $H$ be a subgraph of a regular $(\Delta+1)$-graph $H'$ such that
  \begin{equation}\label{eq:deltas}
  	\Delta(G)+1 \ge \Delta(H')-\Delta(H).
  \end{equation}
  Then, 
  $$
      \chi_i(G \cartp H) \le \Delta(G \cartp H) + 2\,.
  $$
\end{theorem}

\begin{proof}
	Let $c$ be an optimal incidence coloring of $G$ with colors from 
	$A=\{0,\dots,\Delta(G)\}$ and let $d'$ be an optimal incidence coloring of $H'$ 
	with colors from $B=\{t,\dots,\Delta(H')+t\}$, where 
	$$
		t=\Delta(G)+\Delta(H)-\Delta(H')+1
	$$
	is an \emph{offset} of the coloring $d'$. Note that $\Delta(H')+t=\Delta(G \cartp H)+1$ and \eqref{eq:deltas} implies $t\ge 0$. Since $d'$ is optimal, 
	every vertex $v$ of $H'$ has a full spectrum, i.e. $S_{d'}(v)=B$. Let $d$ denote the restriction of $d'$ to the subgraph $H$. Clearly, $d$ is an 
	incidence coloring of $H$ with at most $|B|=\Delta(H')+1$ colors. 

	Let $C=A\cap B=\{t,\dots,\Delta(G)\}$ denote the set of \emph{overlapping colors} between the colorings $c$ and $d$. Note that 
	$$
		|C|=\Delta(G)-t+1=\Delta(H')-\Delta(H).
	$$
	Since $H$ is a subgraph of a regular ($\Delta+1$)-graph $H'$ and $d$ is a restriction of an optimal coloring $d'$, it follows by Observation~\ref{obs:reg} that the 
	$d$-spectrum of each vertex $v$ of $H$ is minimal, i.e.
	$$
		|S_d(v)|=d_H(v)+1\le \Delta(H)+1.
	$$
	This ensures that for the set $M(v)=B\setminus S_d(v)$ of \emph{missing colors} at the vertex $v$ we have
	$$
		|M(v)|=|B|-|S_d(v)|\ge \Delta(H')-\Delta(H)=|C|.
	$$
	The first equality holds since obviously $S_d(v)\subseteq B$. 
	Hence, there exists an injective mapping $g_v\colon C \to M(v)$. 
	We may choose such $g_v$ arbitrarily.

	Then, an incidence coloring $f$ of $G \cartp H$ with at most $\Delta(G \cartp H)+2$ colors is constructed as follows. 	
	For every pair of vertices 
	$u\in V(G)$, $v\in V(H)$, and edges $uu'\in E(G)$, $vv'\in E(H)$ we define
	\begin{align*}
	f\big((u,v), (u,v)(u,v')\big)&=d(v,vv'),\quad\text{and} \\
	f\big((u,v), (u,v)(u',v)\big)&=
	\begin{cases}
		\ c(u, uu') & \text{if }c(u,uu')\notin C,\\
		\ g_v(c(u, uu')) & \text{if }c(u,uu')\in C. \end{cases}
	\end{align*}
	Informally, an $H$-incidence of the vertex $(u,v)$ receives in $G \cartp H$ the same color as in the coloring of $H$, and a 
	$G$-incidence receives the same color as in the coloring of $G$ if it is a non-overlapping color, otherwise it receives one 
	of the missing colors of $d$ at the vertex $v$ determined by $g_v$ as a replacement instead of this overlapping color. 

	It remains to verify that the coloring $f$ is a valid incidence coloring. First, clearly no two $H$-incidences in the same $H$-f\mbox{}iber 
	collide as their colors are inherited from the incidence coloring $d$ of $H$. 

	Second, no two $G$-incidences in the same $G$-f\mbox{}iber with respect to $v\in V(H)$ collide as their colors are inherited 
	from the coloring $c$ of $G$ up to replacing the overlapping colors via the (injective) mapping $g_v$. Observe also that two $G$-incidences 
	from different $G$-f\mbox{}ibers (as well as two $H$-incidences from different $H$-f\mbox{}ibers) never collide as they are simply non-adjacent. 
	\begin{figure}[ht]
	\centerline{\includegraphics[scale=0.8]{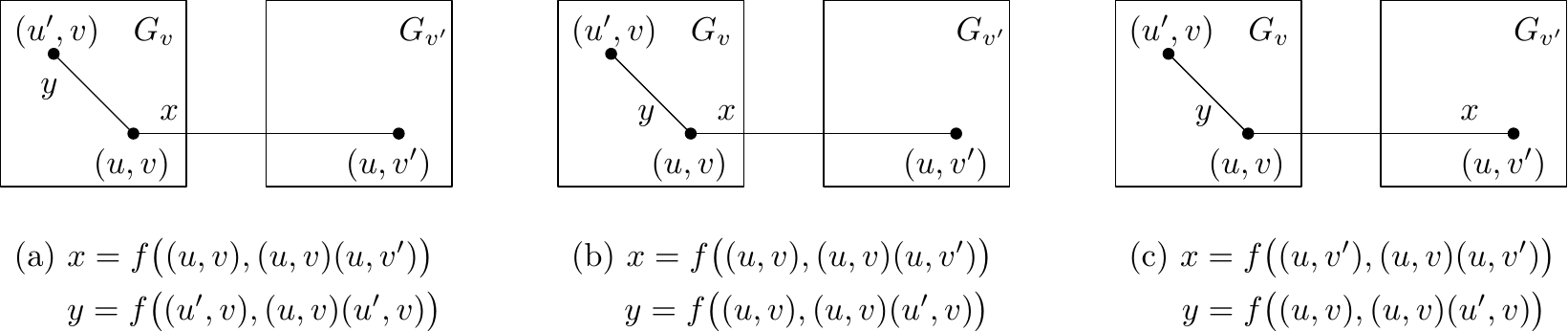}}
	\caption{Interference between $G$-incidences and $H$-incidences in 
	$G \cartp H$: (a) adjacent incidences at different vertices from the same $G$-f\mbox{}iber, (b) adjacent incidences at the same vertex, 
	(c) adjacent incidences at different vertices in different $G$-f\mbox{}ibers.}
	\label{fig:GHinc}
	\end{figure}

	Last, we consider interference between $G$-incidences and $H$-incidences. There are three cases as depicted in Fig.~\ref{fig:GHinc}. 
	In all these cases the color $x$ of the $H$-incidence is included in the $d$-spectrum of $v$, i.e. $x\in S_d(v)$. 
	Furthermore, by the definition of $f$, the color $y$ of the $G$-incidence is either a non-overlapping color from $A$ or 
	one of the missing colors from $M(v)$, hence consequently, $y\notin S_d(v)$. Therefore the colors $x$ and $y$ are distinct as well.
\end{proof}

It is worth noting that the condition \eqref{eq:deltas} can be replaced as follows since every graph is a subgraph of the complete graph, which is a regular $(\Delta+1)$-graph.
\begin{corollary}
	\label{cor:deg}
	Let $G$ be a ($\Delta+1$)-graph and let $H$ be a graph such that
	$$ 
		\Delta(G)+1 \ge |V(H)|-1-\Delta(H)\,.
	$$
	Then, $\chi_i(G \cartp H) \le \Delta(G \cartp H) + 2$.
\end{corollary}

From Theorem~\ref{thm:product}, we can derive the result for hypercubes.
\begin{corollary}
	\label{cor:hyper}
    For every $n\ge 1$,
    $$
		\chi_i(Q_n) = 
		\begin{cases}
			\ n + 1 & \text{if }n = 2^m - 1\text{ for some integer }m\ge 0,\\
			\ n + 2 & \text{otherwise.} 
		\end{cases}
    $$
\end{corollary}

\begin{proof}
	Let $n=2^m-1+k$ where $m$ and $k$ are integers such that $m\ge 0$ and $0 \le k \le 2^m-1$. 
	If $k=0$, then there exists a partition of $V(Q_n)$ into $n+1$ dominating sets obtained by a well-known Hamming code and its cosets~\cite{MacSlo77}. 
	Hence by Theorem~\ref{thm:partition}, $\chi_i(Q_n)=n+1$.
    
	If $k\ge 1$, we may represent $Q_n$ as the Cartesian product $G \cartp H$ of $G=Q_{2^m-1}$ and $H=Q_{k}$. 
	Since $G$ is a $(\Delta+1)$-graph, $H$ is a subgraph of a regular $(\Delta+1)$-graph $H'=Q_{2^m-1}$, and 
	$$
		2^m=\Delta(G)+1\ge \Delta(H')-\Delta(H)=2^m-1-k
	$$
	satisfies the condition~\eqref{eq:deltas}, we have by Theorem~\ref{thm:product} that $\chi_i(Q_n)\le \Delta(Q_n)+2=n+2$. 
	Finally, as $Q_n$ in this case does not have a partition into $n+1$ dominating sets~\cite{MacSlo77}, 
	this upper bound is tight by Theorem~\ref{thm:partition}; that is, $\chi_i(Q_n)=n+2$.    
\end{proof}

\section{Cartesian products with $2$-permutable graphs}
\label{sec:2}

In this section, we introduce two classes of graphs, $2$-permutable and $2$-adjustable graphs. We determine their basic properties
and show that a Cartesian product of a $2$-permutable graph and a $2$-adjustable graph is a $(\Delta+2)$-graph.

A \textit{homomorphism} $f$ of a graph $G$ to a graph $H$ is a mapping
$$
	f\,:\,V(G)\rightarrow V(H)
$$ 
such that if $uv \in E(G)$, then $f(u)f(v) \in E(H)$. 
A homomorphism is \textit{locally injective} if 
$f(u) \neq f(w)$, for every $v \in V(G)$ and every pair $vu,vw \in E(G)$;
i.e. $f$ is injective on the set $N(v)$ of neighbors of any vertex $v$
(see~\cite{FiaKra08} for more details).
A key property of locally injective homomorphisms in the context of incidence colorings is that they preserve adjacencies of incidences.
In his thesis, Duffy~\cite{Duf15} deduced the following.
\begin{theorem}[Duffy, 2015]
	\label{thm:duf}
	Let $G$ and $H$ be simple graphs such that $G$ admits a locally injective homomorphism to $H$. Then
	$$
		\chi_i(G) \le \chi_i(H)\,.
	$$
\end{theorem}

Using the equality in~\eqref{eq:square} and the properties of locally injective homomorphisms,
even a stronger statement holds, in the case when the graph $H$ from Theorem~\ref{thm:duf} is a complete graph. 
\begin{proposition}
	A graph $G$ admits a $(k,1)$-incidence coloring if and only if it admits a locally injective homomorphism to $K_{k}$.
\end{proposition}

On the other hand, we do not know what happens in the case with $p \ge 2$.
\begin{question}
	Does, for a given $k$ and $p \ge 2$, a finite family of graphs $\mathcal{I}_{k,p}$ exist such that
	every connected $(k,p)$-graph admits a locally injective homomorphism to some graph $I \in \mathcal{I}_{k,p}$?
\end{question}

Let $K_{2n}^{-}$ be a complete graph on $2n$ vertices without a perfect matching.
A connected $2d$-regular graph $G$ is \textit{$2$-permutable} if it admits a locally injective homomorphism to $K_{2d+2}^-$.
This in particular means that $G$ is $(2d + 2)$-partite (with partition sets $P_1,\dots,P_{2d+2}$),
and for every $i$, $1 \le i \le 2d+2$, there exists $\overline{i}$ such that there are no edges between $P_i$ and $P_{\overline{i}}$.
Moreover, every $v \in P_i$ has exactly one neighbor in $P_j$, for every $j$, $j \notin \set{i,\overline{i}}$.

By the definition it immediately follows that every $2$-permutable graph $G$ is a $(\Delta+2,1)$-graph. 
Indeed, we obtain a coloring $c$ of $G$ in the following way: 
let $v_1,\dots,v_{\Delta(G)+2}$ be the vertices of $K_{\Delta(G)+2}^-$ and let $h$ be a locally injective homomorphism 
from $G$ to $K_{\Delta(G)+2}^-$. For every incidence $(u,uw) \in I(G)$ such that $h(w)=v_j$,  
let $c((u,uw)) = j$. Hence, for a vertex $u \in G$ such that $h(u) = v_j$, we have $S_c^1(u) = \set{j}$ and 
there is precisely one color $\mu_c(u) = \overline{j}$ missing in the spectrum $S_c(u)$.

The simplest examples of $2$-permutable graphs are cycles on $4k$ vertices. An exhaustive computer search showed that there are precisely
$13$ non-isomorphic $4$-regular graphs of order $12$ (out of $1544$) which are $2$-permutable.
On the other hand, there exist $(\Delta+2,1)$-graphs which are not $2$-permutable; consider e.g. a cycle on seven vertices.
\begin{theorem}
	\label{thm:prism}
	Let $G$ be a $2$-permutable graph. Then
	$$
		\chi_i(G \cartp K_2) = \Delta(G \cartp K_2) + 1\,.
	$$
\end{theorem}
\begin{proof}
	Let $G_1$ and $G_2$ be the two $G$-f\mbox{}ibers of $G \cartp K_2$, and let $c$ be a $(\Delta(G)+2,1)$-incidence coloring of $G$ as described above.
	Let $\overline{c}$ be a coloring obtained from $c$ by replacing every color $i$ by the color $\overline{i}$. Color the incidences of $G_1$
	by the colors given by $c$, and the incidences of $G_2$ by the colors given by $\overline{c}$. We complete the coloring of $G \cartp K_2$ by 
	coloring the $K_2$-incidences. For every $v_i \in G$ let $v_i^1$ and $v_i^2$ be the vertex associated with $v_i$ in $G_1$ and $G_2$, respectively.
	Color $(v_i^1,v_i^1v_i^2)$ by $\overline{i}$, and $(v_i^2,v_i^1v_i^2)$ by $i$. As $\overline{i}$ is not in $S_c(v_i^1)$ and $i$ is not in 
	$S_{\overline{i}}(v_i^2)$, the obtained coloring is indeed a $(\Delta(G \cartp K_2)+1)$-incidence coloring of $G \cartp K_2$.
\end{proof}
Note that the inverse statement of Theorem~\ref{thm:prism} does not hold; a Cartesian product of the dodecahedron and $K_2$ is a $(\Delta+1)$-graph,
while the dodecahedron is not $2$-permutable, as it is a cubic graph. However, it does hold if $G$ is $2$-regular.
\begin{proposition}
	A connected $2$-regular graph $G$ is $2$-permutable if and only if
	$$
		\chi_i(G \cartp K_2) = \Delta(G \cartp K_2) + 1\,.
	$$
\end{proposition}

\begin{proof}
	We only need to prove the left implication.
	Let $G$ be a cycle such that 
	$\chi_i(G \cartp K_2) = \Delta(G \cartp K_2) + 1$.
	Let $G_1$ and $G_2$ be the two $G$-f\mbox{}ibers of $G \cartp K_2$, and let $c$ be its $(\Delta(G)+1)$-incidence coloring.
	For every $v \in G$ let $v^1$ and $v^2$ be the vertex associated with $v$ in $G_1$ and $G_2$, respectively.
	A $4$-partition of $G$ is then obtained by adding a vertex $v$ in $P_i$, if $c((v^2,v^1v^2)) = i$, where $1 \le i \le 4$. 
	As every incoming incidence of a vertex $v$ has the same color, we immediately infer that coloring of the incidences of $v$ 
	uniquely determines the colors of all the other incidences. 
	Consequently, every four consecutive vertices of $G$ belong to distinct partition 
	sets, and hence $G$ is isomorphic to the cycle $C_{4k}$, for some positive integer $k$.
\end{proof}

A (non-regular) graph $G$ is \textit{sub-$2$-permutable} if it admits a locally injective homomorphism to $K_{\Delta(G)+2}^-$. The following
is an immediate consequence of Theorem~\ref{thm:prism}.
\begin{corollary}
	Let $G$ be a sub-$2$-permutable graph. Then
	$$
		\chi_i(G \cartp K_2) = \Delta(G \cartp K_2) + 1\,.
	$$
\end{corollary}

\medskip
An incidence coloring of a graph $G$ is \textit{adjustable} if there exists a pair of colors $x$ and $y$ such that there is no 
vertex $v \in V(G)$ with $x,y \in S_0(v)$. We denote the colors $x$ and $y$ as \textit{free colors}. 
A graph $G$ is \textit{$2$-adjustable} if it admits an adjustable $(\Delta(G)+2)$-incidence coloring. 

Notice that all $(\Delta+1)$-graphs and $(\Delta+1)$-graphs together with a matching are $2$-adjustable; in the former case we have an extra color 
which is not used at all, and in the latter, the same two colors are used on the incidences corresponding to the matching edges.
For example, all cycles, complete bipartite graphs, and prisms $C_{6n} \cartp K_2$ are $2$-adjustable. 
In fact, in Proposition~\ref{prop:adj} we show that a big subclass of $2$-adjustable graphs can be classified in terms of homomorphisms.
By $\mathring{K}_n$ we denote the complete graph of order $n$ with a loop (see Fig.~\ref{fig:adjust} for an example).
\begin{figure}[ht]
	$$
		\includegraphics[scale=0.8]{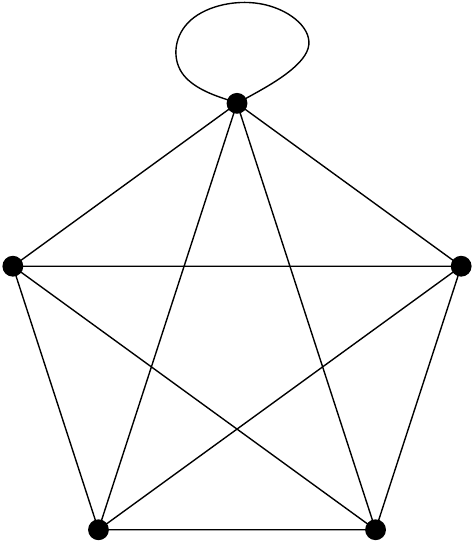}
	$$
	\caption{The graph $\mathring{K}_5$.}
	\label{fig:adjust}
\end{figure}
\begin{proposition}
	\label{prop:adj}
	If a graph $G$ admits a locally injective homomorphism to $\mathring{K}_{\Delta(G)+1}$, then $G$ is $2$-adjustable.
\end{proposition}
\begin{proof}
	Let $G$ admit a locally injective homomorphism $f$ to $\mathring{K}_{\Delta(G)+1}$ and let $v_0, v_1,\dots,v_{\Delta}$ be the vertices of $\mathring{K}_{\Delta(G)+1}$,
	where $v_0$ is the vertex incident with a loop. Let $g$ be a function such that $g(v) = i$ if $f(v) = v_i$, for every $v \in V(G)$.
	Notice that the vertices $f^{-1}(v_0)$ of $G$ mapped to $v_0$ induce a subgraph $H$ of maximum degree $1$ in $G$,
	i.e. every vertex of $G$ has at most one neighbor $u$ such that $f(u)=v_0$. Let $d$ be a proper vertex coloring of $H$ with $2$ colors, say $\Delta(G)+1$ and $\Delta(G) + 2$. 
	An adjustable coloring $c$ of $G$ with free colors $\Delta(G)+1$ and $\Delta(G)+2$ can then be constructed as follows. For each $v \in V(G)$ let
    $$
		c((v, uv)) = 
		\begin{cases}
			\ g(u) & \text{if }g(u) > 0,\\
			\ d(u) & \text{otherwise.} 
		\end{cases}
    $$	
    It is straightforward to verify that $c$ is an incidence coloring, in fact, $c$ is even a $(\Delta(G)+2,1)$-incidence coloring.
\end{proof}

Again, the inverse is not true in general. A cycle on $5$ vertices is $2$-adjustable (see Fig.~\ref{fig:fivecycle}), but does not admit a locally injective 
homomorphism to $\mathring{K}_3$.
\begin{figure}[ht]
	$$
		\includegraphics[scale=0.8]{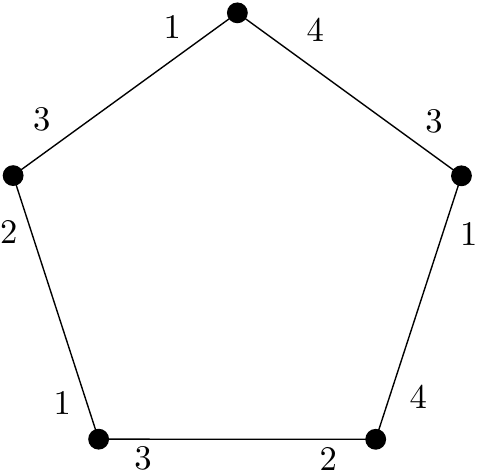}
	$$
	\caption{An adjustable incidence coloring of $C_5$ with four colors, with the free colors $1$ and $2$.}
	\label{fig:fivecycle}
\end{figure}

Now, we are ready to state the main theorem of this section.
\begin{theorem}
	\label{thm:peradj}
	Let $G$ be a sub-$2$-permutable graph and let $H$ be a $2$-adjustable graph. Then
	$$
		\chi_i(G \cartp H) \le \Delta(G \cartp H) + 2\,.
	$$
\end{theorem}	

\begin{proof}
	Let $c$ and $\overline{c}$ be two $(\Delta(G)+2)$-incidence colorings of $G$ with the colors $1,\dots,\Delta(G)+2$ as described in the proof of
	Theorem~\ref{thm:prism}. Let $d$ be an adjustable incidence coloring of $H$ with the colors $\Delta(G)+1,\Delta(G)+2,\dots,\Delta(G)+\Delta(H)+2$, 
	where $x=\Delta(G)+1$ and $y=\Delta(G)+2$ are free colors in $d$.	
	
	Now, we construct an incidence coloring $f$ of $G \cartp H$ with $\Delta(G)+\Delta(H)+2$ colors. 
	For every pair of vertices $u\in V(G)$, $v\in V(H)$, and edges $uu'\in E(G)$, $vv'\in E(H)$ we define
	\begin{align*}
		f\big((u,v), (u,v)(u',v)\big)&=
		\begin{cases}
			\ c(u, uu') & \text{if }y \notin S_d^0(v),\\
			\ \overline{c}(u, uu') & \text{if }y \in S_d^0(v)
		\end{cases} \\
	\end{align*}		
	and
	\begin{align*}
		f\big((u,v), (u,v)(u,v')\big)&=
		\begin{cases}
			\ d(v,vv') & \text{if } d(v,vv') \notin \set{x,y},\\
			\ \mu_c(u) & \text{if } d(v,vv') = x, \\
			\ \mu_{\overline{c}}(u) & \text{if } d(v,vv') = y\,.
		\end{cases}				
	\end{align*}
	Recall, $\mu_c(v)$ is the color missing in the spectrum $S_c(v)$.
	Informally, we color the incidences in $G$-f\mbox{}ibers by the colorings $c$ and $\overline{c}$, and the incidences in $H$-f\mbox{}ibers by the coloring $d$,
	where the free colors $x$ and $y$ are being replaced by the colors not used by $c$ and $\overline{c}$. See Fig.~\ref{fig:c4c5} for an example of such 
	a coloring of $C_4 \cartp C_5$.
	
	It remains to prove that $f$ is indeed an incidence coloring of $G \cartp H$. Clearly, with an exception of $H$-incidences receiving the colors $x$ and $y$
	in $d$, there are no conflicts in $f$. So, let $((u,v),(u,v)(u,v'))$ be an $H$-incidence colored by $\mu_c(u)$	(resp. $\mu_{\overline{c}}(u)$). There
	is no conflict in $H_u$ (since there would be two adjacent incidences colored by $x$ in $H$), neither in $G_v$ (by the definition of $\mu_c(u)$, resp. $\mu_{\overline{c}}(u)$). 
	This completes the proof.
\end{proof}

\begin{figure}[ht]
	$$
		\includegraphics{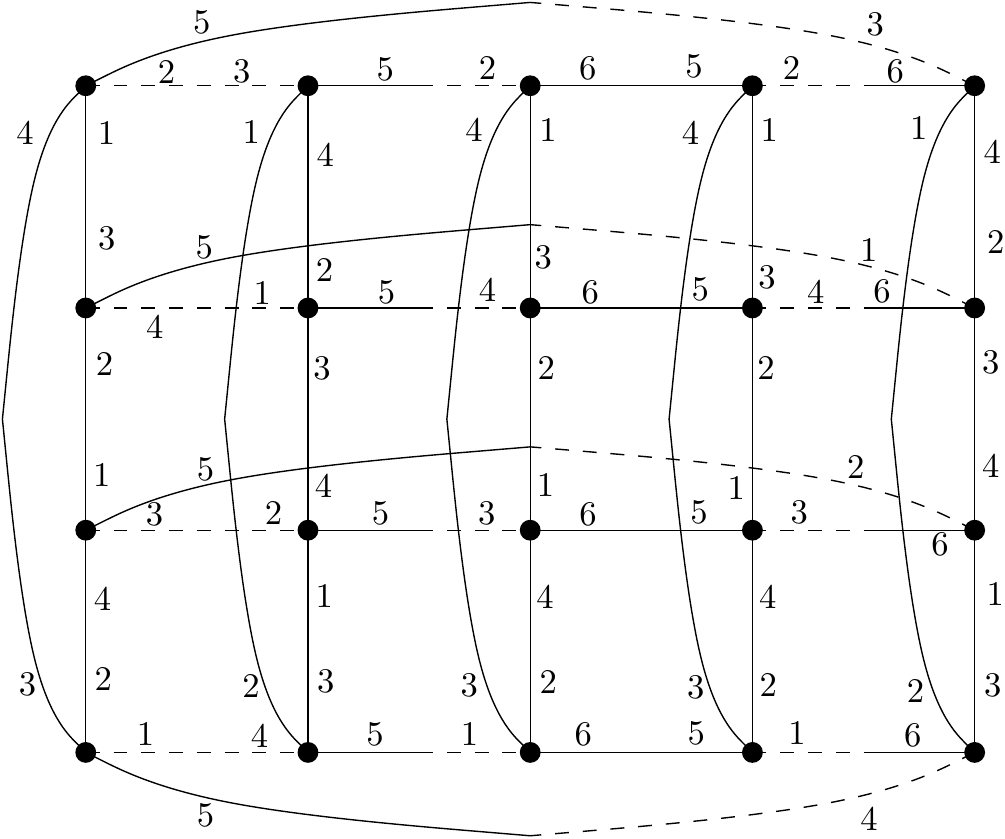}
	$$
	\caption{An example of an incidence coloring of $C_4 \cartp C_5$ obtained from incidence colorings of $2$-permutable graph $C_4$ 
	and $2$-adjustable graph $C_5$ as described in the proof of Theorem~\ref{thm:peradj}. The dashed lines denote the $H$-incidences being recolored.}
	\label{fig:c4c5}
\end{figure}

For example, as a corollary of Theorem~\ref{thm:peradj} we immediately infer that the incidence chromatic number of the Cartesian product 
of cycles $C_{4n}$ and $C_m$ is a $(\Delta+2)$-graph (in fact the statement holds for the Cartesian products of cycles of arbitrary lengths as shown 
by Sopena and Wu~\cite{SopWu13}).

\section{Conclusion}
\label{sec:con}

Clearly, Conjecture~\ref{conj:main} is not valid for all Cartesian products of graphs;
consider simply a Cartesian product of a Paley graph and $K_2$. 
We have shown that it holds if one of the factors is a $(\Delta+1)$-graph and the maximum degree of the second is big enough (see Corollary~\ref{cor:deg}). 
We believe the following also holds.
\begin{conjecture}
	Let $G$ be a $(\Delta+1)$-graph and $H$ be a graph with $\chi_i(H) \le \Delta(H) + 2$.
    Then,
    $$
    	\chi_i(G \cartp H) \le \Delta(G \cartp H) + 2\,.
    $$
\end{conjecture}

We also ask if the condition on a graph $G$ (the first factor) can be somehow relaxed. As Sopena and Wu~\cite{SopWu13} showed, 
every toroidal grid $T_{m,n}$, i.e. a Cartesian product of cycles $C_m$ and $C_n$, admits an incidence coloring with at most $\Delta(T_{m,n}) + 2$ colors. 
\begin{question}
	Do there exist graphs $G$ and $H$ with $\chi_i(G) = \Delta(G) + 2$ and $\chi_i(H) = \Delta(H) + 2$ such that
	$\chi_i(G \cartp H) > \Delta(G\cartp H) + 2$? 
\end{question}
On the other hand, when both factors are $(\Delta+1)$-graphs, we may also ask:
\begin{question}
	When is the Cartesian product of two $(\Delta+1)$-graphs also a $(\Delta+1)$-graph? 
\end{question}

In this paper, we have introduced two recoloring techniques. We always color the f\mbox{}ibers of the product in the same way (up to a permutation of colors in the f\mbox{}ibers of the same factor) 
and then try to recolor some colors in the f\mbox{}ibers of one factor using some available colors of the second factor. 
These techniques are somehow limited as we preserve the structure of both initial colorings,
whereas it may be more efficient to use all the available colors everywhere. In~\cite{SopWu13}, the authors use pattern tiling, but their graphs are well structured in 
that case. Additional techniques of recoloring, but still not relying to the structure of the factors too much, would surely contribute to the field.

\paragraph{Acknowledgment.} 
The authors thank two anonymous referees for their remarks and careful reading of the manuscript.
This research was supported by the Czech Science Foundation grant GA14--10799S and Slovenian Research Agency Program P1--0383.
The second author also acknowledges partial support by the National Scholarship Programme of the Slovak Republic.

\bibliographystyle{plain}

\bibliography{mainBib}

\end{document}